\documentclass[11pt]{amsart}
\usepackage{amsfonts}
\usepackage{amssymb}
\usepackage{amsmath}
\usepackage{amsthm}
\theoremstyle{plain}
\newtheorem{thm}{Theorem}[section]
\newtheorem{lem}[thm]{Lemma}
\theoremstyle{definition}

\newtheorem{rem}[thm]{Remark}

\numberwithin{equation}{section}

\begin{document}
\title[$(p_1,p_2)$-Laplacian system]{Multiple Positive solutions of a $(p_1,p_2)$-Laplacian system with
nonlinear BCs}
\author{Filomena Cianciaruso}%
\address{Filomena Cianciaruso, Dipartimento di Matematica e Informatica, Universit\`{a} della
Calabria, 87036 Arcavacata di Rende, Cosenza, Italy}
\email{cianciaruso@unical.it}
\author{Paolamaria Pietramala}%
\address{Paolamaria Pietramala, Dipartimento di Matematica e Informatica, Universit\`{a}
della Calabria, 87036 Arcavacata di Rende, Cosenza, Italy}
\email{pietramala@unical.it} \subjclass[2010]{Primary 45G15,
secondary  34B18} \keywords{Fixed point index, cone, positive
solution, p-laplacian, system, nonlinear boundary conditions.}
\begin{abstract}
Using the theory of fixed point index, we discuss  existence, non-existence, localization and multiplicity of positive solutions for a $(p_1,p_2)$-Laplacian system with nonlinear  Robin and/or Dirichlet  type boundary conditions.
We give an example to illustrate our theory.
\end{abstract}
\maketitle
\section{Introduction}
In the remarkable paper \cite{wa} Wang proved the existence of one positive solution of following one-dimensional $p$-Laplacian equation
\begin{equation}\label{kawa}  
(\varphi_{p}(u^{\prime}))'(t)+g(t)f(u(t))=0,\ t \in (0,1), \\
\end{equation}
subject to one of the following three pair of nonlinear boundary conditions (BCs)
\begin{equation*}  
u'(0)=0\,,\,\,u(1)+B_1\left(u'(1)\right)=0,
\end{equation*}
\begin{equation*}  
u(0)=B_0\left(u'(0)\right)\,,\,\,u'(1)=0.
\end{equation*}
\begin{equation*}  
u(0)=B_0\left(u'(0)\right)\,,\,\,u(1)+B_1\left(u'(1)\right)=0.
\end{equation*}
The results of \cite{wa} were extended by Karakostas \cite{kara} to the context of deviated arguments. In both cases,
the existence results are obtained via a careful study of an associated integral operator combined with the use of  the Krasnosel'ski\u\i{}-Guo Theorem on cone compressions and cone expansions. 

The Krasnosel'ski\u\i{}-Guo Theorem, more in general, topological methods are a commonly used tool in the study of existence of
 positive solutions for the $p$-Laplacian equation \eqref{kawa} subject to different BCs. This is an active area of research, for example, homogeneous  Dirichlet  BCs have been studied in
 \cite{aga-lu-ore, bai-chen, inmapr, kim, lu-ore-zho,  sim-lee, wa-zha, ya-ore}, homogeneous Robin  BCs in \cite{lu-ore-zho, wa-zha, ya-ore}, non local BCs of Dirichlet  type  in \cite{ave-he, ba-dje-mo, bai-fang, calve,  fe-ge-jian, he-ge,  kara2, wa, wa-ge, zha}  and nonlocal BCs of Robin type in \cite{he-ge,  li-shen, ma-du,  wa-ge2, wa-ho, zha}.
 
Here we study the the one-dimensional $(p_1,p_2)$-Laplacian system
\begin{equation}  \label{lap}
\begin{array}{c}
(\varphi_{p_1}(u^{\prime}))'(t)+g_1(t)f_{1}(t, u(t),v(t))=0,\ t \in (0,1), \\
(\varphi_{p_2}(v^{\prime}))'(t)+g_2(t)f_{2}(t, u(t),v(t))=0,\ t \in (0,1), \\
\end{array}%
\end{equation}
with $\varphi_{p_i}(w)=\vert w\vert^{p_i-2}w$, subject to the nonlinear boundary conditions (BCs)
\begin{equation}  \label{bc}
u'(0)=
0\,,\,\,u(1)+B_1\left(u'(1)\right)=0,\,\,\,\,v(0)=B_2\left(v'(0)\right),\,\,v(1)=0.
\end{equation}

The existence of \emph{positive} solutions of systems of  equations of the type \eqref{lap} has been widely studied, see for example \cite{chen-lu,  la-zha, lee-lee, xu-lee}
under homogeneous Dirichlet BCs  and \cite{inmapr, je-pre, pra-ku-mu, su-we-xu,ya} with homogeneous Robin or Neumann BCs. 
For earlier contributions on problems with nonlinear BCs we refer
to \cite{Goodrich1, Goodrich2, gi-caa, gifmpp-cnsns, gipp-cant, gipp-nonlin, gipp-mmas, kara, li-shen, paola, wa} and references therein.

We improve and complement the previous results  in several
directions:  we
obtain \emph{multiplicity}  results for $(p_1, p_2)$-Laplacian \emph{system}
subject to \emph{nonlinear} BCs, we allow different growths in the nonlinearities $f_1$ and $f_2$ and we also discuss
non-existence results. Finally we  illustrate in an example that all the constants that occur in our results can be computed.

Our approach  is
to seek  solutions of the system \eqref{lap}-\eqref{bc} as fixed points  of a suitable integral operator.
We make use of the classical fixed point index theory and benefit of ideas from the papers \cite{gipp-nonlin, gipp-nodea, kara, wa}.\\

\section{The system of  integral equations}\label{sec2}
We recall that a \emph{cone} $K$ in a Banach space $X$  is a
closed convex set such that $\lambda \, x\in K$ for $x \in K$ and
$\lambda\geq 0$ and $K\cap (-K)=\{0\}$.\\
If $\Omega$ is a open bounded subset of a cone $K$ (in the relative topology) we denote by $\overline{\Omega}$ and $\partial \Omega$
the closure and the boundary relative to $K$. When $\Omega$ is an open bounded subset of $X$ we write $\Omega_K=\Omega \cap K$, an open subset of $K$.\\
The following Lemma summarizes some classical results regarding the fixed point index, for more details see~\cite{Amann-rev, guolak}.
\begin{lem}
Let $\Omega$ be an open bounded set with $0\in \Omega_{K}$ and $\overline{\Omega}_{K}\ne K$. Assume that $F:\overline{\Omega}_{K}\to K$ is
a compact map such that $x\neq Fx$ for all $x\in \partial \Omega_{K}$. Then the fixed point index $i_{K}(F, \Omega_{K})$ has the following properties.
\begin{itemize}
\item[(1)] If there exists $e\in K\setminus \{0\}$ such that $x\neq Fx+\lambda e$ for all $x\in \partial \Omega_K$ and all $\lambda
>0$, then $i_{K}(F, \Omega_{K})=0$.
\item[(2)] If  $\mu x \neq Fx$ for all $x\in \partial \Omega_K$ and for every $\mu \geq 1$, then $i_{K}(F, \Omega_{K})=1$.
\item[(3)] If $i_K(F,\Omega_K)\ne0$, then $F$ has a fixed point in $\Omega_K$.
\item[(4)] Let $\Omega^{1}$ be open in $X$ with $\overline{\Omega^{1}}\subset \Omega_K$. If $i_{K}(F, \Omega_{K})=1$ and
$i_{K}(F, \Omega_{K}^{1})=0$, then $F$ has a fixed point in $\Omega_{K}\setminus \overline{\Omega_{K}^{1}}$. The same result holds if
$i_{K}(F, \Omega_{K})=0$ and $i_{K}(F, \Omega_{K}^{1})=1$.
\end{itemize}
\end{lem}
To the system (\ref{lap})-(\ref{bc}) we associate the following system of integral equations, which is constructed in similar manner as in \cite{wa}, where the case of 
a single equation is studied. 
\begin{gather}
\begin{aligned}\label{syst}
&u(t)=\int_{t}^{1}\varphi_{p_1}^{-1}\Bigl(\int_{0}^{s}
g_1(\tau)f_1(\tau,u(\tau),v(\tau))\,d\tau \Bigr)
\,ds\\&\,\,\,\,\,\,\,\,\,\,\,\,\,\,\,\,\,\,\,\,\,\,\,\,\,\,\,\,\,\,\,\,+B_1\left(\varphi_{p_1}^{-1}\left(\int_{0}^{1}g_1(\tau)
f_1(\tau,u(\tau),v(\tau))\,d\tau \right)\right),\,\,\,\,\,0\le
t\le 1,\cr &v(t)=
\begin{cases} \int_{0}^{t}
\varphi_{p_2}^{-1}\left(\int_{s}^{\sigma_{u,v}}g_2(\tau)f_2(\tau,u(\tau),v(\tau))d\tau\right)\,ds\\ \,\,\,\,\,\,\,\,\,\,\,\,\,\,\,+B_2\left(\varphi_{p_2}^{-1}\left(\int_{0}^{\sigma_{u,v}}g_2(\tau)
f_2(\tau,u(\tau),v(\tau))\,d\tau\right)\right),
\,\,\,\,\,\,\,\,0\le t\le \sigma_{u,v},\cr
\int_{t}^{1}\varphi_{p_2}^{-1}\left(\int_{\sigma_{u,v}}^{s}g_2(\tau)f_2(\tau,u(\tau),v(\tau))d\tau\right)\,ds,\,\,\,\,\,\,\,\,\,\,\,\,\,\,\,\,\,\,\,\,\, \sigma_{u,v}\leq t\leq 1,\cr \end{cases} %
\end{aligned}
\end{gather}
where $\varphi_{p_i}^{-1}(w)=\vert w\vert^{\frac{1}{p_i-1}}$ sgn
$w$ and $\sigma_{u,v}$ is the smallest  solution $x\in [0,1]$ of the
equation
\begin{eqnarray*}
&&\int_{0}^{x}\varphi_{p_2}^{-1}\left(\int_{s}^{x}g_2(\tau)
f_2(\tau,u(\tau),v(\tau))d\tau\right)\,ds
+B_2\left(\varphi_{p_2}^{-1}\left(\int_{0}^{x}g_2(\tau)
f_2(\tau,u(\tau),v(\tau))\,d\tau\right)\right)\\
&&\qquad\qquad= \int_{x}^{1} \varphi_{p_2}^{-1}\left(\int_{x}^{s}
g_2(\tau)f_2(\tau,u(\tau),v(\tau))d\tau\right)\,ds.
\end{eqnarray*}
By a \emph{solution}  of (\ref{lap})-(\ref{bc}), we mean a solution of the system \eqref{syst}.

In order to utilize the fixed point index theory we
state the following assumptions on the terms that occur in the
system \eqref{syst}.
\begin{itemize}
\item [$(C1)$] For every $i=1,2$, $f_i: [0,1]\times [0,\infty)\times [0,\infty) \to [0,\infty)$ satisfies Carath\'{e}odory conditions, that is, $f_i(\cdot,u,v)$ is measurable for each fixed $(u,v)$ and $f_i(t,\cdot,\cdot)$ is continuous for almost every (a.e.) $t\in [0,1]$, and for each $r>0$ there exists $\phi_{i,r} \in L^{\infty}[0,1]$ such that{}
\begin{equation*}
f_i(t,u,v)\le \phi_{i,r}(t) \;\text{ for } \; u,v\in [0,r]\;\text{ and a.\,e.} \; t\in [0,1].
\end{equation*}%
{}
\item [$(C2)$] $g_1\in L^{1}[0,1]$, $g_1\geq 0$ and
\begin{equation*}
0<\int_{0}^{1}\varphi_{p_1}^{-1}\Bigl(\int_{0}^{s} g_1(\tau)\,d\tau \Bigr) \,ds <+\infty.
 \end{equation*}
\item [$(C3)$] $g_2\in L^{1}[0,1]$, $g_2\geq 0$ and
\begin{equation}\label{integ2}
0< \int_{0}^{1/2}\varphi_{p_2}^{-1}\left(\int_{s}^{1/2}g_2(\tau)d\tau\right)\,ds +\int_{1/2}^{1}
\varphi_{p_2}^{-1}\left(\int_{1/2}^{s} g_2(\tau)d\tau\right)\,ds<+\infty.
\end{equation}
\item [$(C4)$] For every $i=1,2$, $B_i: \mathbb R\to \mathbb R$ is a continuous
function and there exist $h_{i1}$, $h_{i2}\geq 0$ such that 
\begin{equation*} \label{B}
h_{i1}\,v\le B_i(v)\le h_{i2}\,v \mbox{ for any }v\geq 0.
\end{equation*}
\end{itemize}

\begin{rem} The condition \eqref{integ2} is weaker than the condition
\begin{equation}\label{Str}
0<\int_{0}^{1} \varphi_{p_2}^{-1}\left(\int_{s}^{1} g_2(\tau)d\tau\right)ds<+\infty.
\end{equation}
In fact, for example, the function
$$ g_2(t)=
\begin{cases}
\frac{1}{(t-1)^2},\,\,\, t\in [0,1/2],\cr
\frac{1}{t^2},\,\,\, t\in (1/2,1],\cr
\end{cases}
$$
satisfies \eqref{integ2} but not satisfies \eqref{Str}.
\end{rem}

\begin{rem}
From  $(C2)$ and $(C3)$ follow that  there exists  $[a_1,b_1]\subset [0,1)$  such that $\int_{a_1}^{b_1} g_1(s)\,ds >0$ and there exists  $[a_2,b_2]\subset (0,1)$  such that $\int_{a_2}^{b_2} g_2(s)\,ds >0$.
\end{rem}

We work in the space $C[0,1]\times C[0,1]$ endowed with the norm
\begin{equation*}
\| (u,v)\| :=\max \{\| u\| _{\infty },\| v\| _{\infty }\},
\end{equation*}%
where $\| w\| _{\infty}:=\max \{| w(t)|,t\in [0,1] \}$.

 Take the cones
\begin{equation*}
K_{1}:=\{w\in C[0,1]: w\geq 0, \mbox { concave and nonincreasing}\},
\end{equation*}%
\begin{equation*}
K_{2}:=\{w\in C[0,1]:w\geq0, \mbox { concave}\}.
\end{equation*}%
It is known (see e.g. \cite{wa}) that
\begin{itemize}
\item for $w\in K_1$ we have $w(t)\geq (1-t)\|w\|_\infty$, for $t\in [0,1]$;
\item  for $w\in K_2$ we have $w(t)\geq \min \{t,1-t\}\|w\|_\infty$, for $t\in [0,1]$.
\end{itemize}
 It follows that the functions in $K_i$ are strictly positive on the sub-interval $[a_i,b_i]$ and in particular we have
\begin{itemize}
\item for $w\in K_1$ we have $\displaystyle\min_{t\in [0,b_1]}w(t)\geq (1-b_1)\|w\|_\infty$;
\item  for $w\in K_2$ we have $\displaystyle\min_{t\in [a_2,b_2]}w(t)\geq \min \{a_2,1-b_2\}\|w\|_\infty$.
\end{itemize}
In the following we make use of the notations:
\begin{equation*}
      c_1:=1-b_1,\,\,\,\,\,\,c_2:=\min \{a_2,1-b_2\}.
 \end{equation*}
Consider now the cone $K$ in $C[0,1]\times C[0,1]$ defined by
\begin{equation*}
K:=\{(u,v)\in K_1\times K_2\}.%
\end{equation*}
For a \emph{positive} solution of the system \eqref{syst} we mean a solution $(u,v)\in K$ of \eqref{syst} such that $\|(u,v)\|> 0$. We seek such solution  as a fixed point of the following operator $T$.\\

Consider the integral operator
\begin{gather}
\begin{aligned}    \label{opT}
T(u,v)(t):=&\left(
\begin{array}{c}
T_1(u,v)(t) \\
T_2(u,v)(t)%
\end{array}
\right),
\end{aligned}
\end{gather}
where
$$
T_1(u,v)(t):=\int_{t}^{1}\varphi_{p_1}^{-1}\Bigl(\int_{0}^{s}g_1(\tau)
f_1(\tau,u(\tau),v(\tau))\,d\tau \Bigr)
\,ds+B_1\left(\varphi_{p_1}^{-1}\big(\int_{0}^{1}g_1(\tau)
f_1(\tau,u(\tau),v(\tau))\,d\tau \big)\right)
$$
and
$$
T_2(u,v)(t):=
\begin{cases} \int_{0}^{t}
\varphi_{p_2}^{-1}\left(\int_{s}^{\sigma_{u,v}}g_2(\tau)f_2(\tau,u(\tau),v(\tau))d\tau\right)\,ds\\ \,\,\,\,\,\,\,\,\,\,\,\,\,\,\,+B_2\left(\varphi_{p_2}^{-1}\left(\int_{0}^{\sigma_{u,v}}g_2(\tau)
f_2(\tau,u(\tau),v(\tau))\,d\tau\right)\right),
\,\,\,\,\,\,\,\,0\le t\le \sigma_{u,v},\cr
\int_{t}^{1}\varphi_{p_2}^{-1}\left(\int_{\sigma_{u,v}}^{s}g_2(\tau)f_2(\tau,u(\tau),v(\tau))d\tau\right)\,ds,\,\,\,\,\,\,\,\,\,\,\,\,\,\,\,\,\,\,\,\,\, \sigma_{u,v}\leq t\leq 1,\cr \end{cases} %
$$
From the definitions,  for every $(u,v)\in K$ we have
$$
\max_{t\in [0,1]}T_2(u,v)(t)=T_2(u,v)(\sigma_{u,v}).
$$
Under our assumptions, we can show that the integral operator $T$ leaves the cone $K$ invariant and is compact.

\begin{lem}\label{compact}
The operator \eqref{opT} maps $K$ into $K$ and is compact.
\end{lem}

\begin{proof}
Take $(u,v)\in K$. Then  we have  $T(u,v)\in K$. Now, we show that the map $T$ is compact.
Firstly, we show that $T$ sends bounded sets into bounded sets. Take $(u,v)\in K$ such that $\| (u,v)\| \leq r$.
Then, for all $t \in [0,1]$  we have
\begin{align*}
T_1(u,v)(t)&= \int_{t}^{1} \varphi_{p_1}^{-1}\left(\int_{0}^{s}
g_1(\tau)f_1(\tau,u(\tau),v(\tau))d\tau\right)ds+ B_1\left(\varphi_{p_1}^{-1}\left(\int_{0}^{1}g_1(\tau)
f_1(\tau,u(\tau),v(\tau))d\tau \right)\right)\\\
&\leq \int_{t}^{1} \varphi_{p_1}^{-1}\left(\int_{0}^{s}
g_1(\tau)\phi_{1,r}(\tau)d\tau\right)ds+h_{12}\varphi_{p_1}^{-1}\left(\int_{0}^{1}g_1(\tau)
f_1(\tau,u(\tau),v(\tau))d\tau
\right) \\
&\leq \int_{t}^{1} \varphi_{p_1}^{-1}\left(\int_{0}^{1}
g_1(\tau)\phi_{1,r}(\tau)d\tau\right)ds+h_{12}\varphi_{p_1}^{-1}\left(\int_{0}^{1}g_1(\tau)
\phi_{1,r}(\tau)d\tau \right)\\
&\leq \int_{0}^{1} \varphi_{p_1}^{-1}\left(\int_{0}^{1}
g_1(\tau)\phi_{1,r}(\tau)d\tau\right)ds+h_{12}\varphi_{p_1}^{-1}\left(\int_{0}^{1}g_1(\tau)
\phi_{1,r}(\tau)d\tau \right)<+\infty.
\end{align*}
We prove now that $T_1$ sends bounded sets into equicontinuous sets. Let $t_1,t_2\in [0,1]$, $t_1<t_2$, $(u,v)\in K$ such that $\| (u,v)\| \leq r$. Then we have
\begin{align*}
|T_1(u,v)(t_1)-T_1(u,v)(t_2)|=&\Bigg|\int_{t_1}^{t_2}\varphi_{p_1}^{-1}\left(\int_{0}^{s} g_1(\tau)f_1(\tau,u(\tau),v(\tau))d\tau\right)\,ds\Bigg|\\
\leq &\Bigg|\int_{t_1}^{t_2}\varphi_{p_1}^{-1}\left(\int_{0}^{1} g_1(\tau)\phi_{1,r}(\tau)d\tau\right)\,ds\Bigg|=C_r|t_1-t_2|.
\end{align*}
Therefore we obtain $|T_1(u,v)(t_1)-T_1(u,v)(t_2)|\rightarrow 0$ when $t_1\rightarrow t_2$. By the Ascoli-Arzel\`{a} Theorem we can conclude that $T_1$ is a compact map. In a similar manner we proceed for $T_{2}(u,v)$.\newline
Moreover, the map $T$ is compact since the components $T_{i}$ are compact maps.
\end{proof}

\section{Existence results}

For our index calculations we use the following (relative) open bounded sets in $K$:
\begin{equation*}
K_{\rho_1,\rho_2} = \{ (u,v) \in K : \|u\|_{\infty}< \rho_1\ \text{and}\ \|v\|_{\infty}< \rho_2\}
\end{equation*}
and
\begin{equation*}
V_{\rho_1,\rho_2} =\{(u,v) \in K: \min_{t\in [a_1,b_1]}u(t)<c_1\rho_1\\ \text{  and}\ \min_{t\in [a_2,b_2]}v(t)<c_2\rho_2\}
\end{equation*}
and if $\rho_1=\rho_2=\rho$ we write simply $K_{\rho}$ and $V_{\rho}$.
 The set $V_\rho$ was introduced in~\cite{df-gi-do} as an extension to the case of systems of a set given by Lan~\cite{lan}.
 The use of different radii, in the spirit of the paper~\cite{gipp-nodea}, allows more freedom in the growth of the nonlinearities. \\

 The following Lemma is similar to the  Lemma $5$ of \cite{df-gi-do}  and therefore its proof is omitted.

\begin{lem}  \label{esca}
  The sets defined above have the following properties:
\begin{itemize}
\item $K_{c_1\rho_1,c_2\rho_2}\subset V_{\rho_1,\rho_2}\subset K_{\rho_1,\rho_2}$.
\item $(w_1,w_2) \in \partial V_{\rho_1,\rho_2}$ \; iff \; $(w_1,w_2)\in K$ and $\displaystyle \min_{t\in [a_i,b_i]} w_i(t)= c_i\rho_i$ for some
$i\in \{1,2\}$ and $\displaystyle \min_{t\in [a_j,b_j]}w_j(t) \le c_j\rho_j$ for $j\neq i$.
\item If $(w_1,w_2) \in \partial V_{\rho_1,\rho_2}$, then for some $i\in\{1,2\}$ $c_i\rho_i \le w_i(t) \le \rho_i$ for each $t \in [a_i,b_i]$
and $\|w_i\|_\infty \leq \rho_i$; moreover for $j\neq i$ we have  $\|w_j\|_\infty \leq \rho_j$.
\end{itemize}
\end{lem}

We firstly prove that the fixed point index is $1$ on the set
$K_{\rho_1,\rho_2}\,$.

\begin{lem}\label{ind1b}
Assume that\ \\
$(\mathrm{I}_{\rho_1,\rho_2 }^{1})$   there exist $\rho_1,\rho_2 >0$ such that for every $i=1,2$
\begin{equation}\label{eqmestt}
 f_i^{\rho_1,\rho_2}  < \varphi_{p_i}(m_i)
\end{equation}{}
where
$$f_{i}^{\rho_1,\rho_2}=\sup \Bigl\{\frac{f_{i}(t,u,v)}{\rho_i^{p_i-1}}:\;(t,u,v)\in \lbrack 0,1]\times [ 0,\rho_1 ]\times [0,\rho_2 ],\Bigr\},$$
$$\frac{1}{m_1}=\int_{0}^{1}\varphi_{p_1}^{-1}\left(\int_0^s
g_{1}(\tau)d\tau\right)\,ds+h_{12}\varphi_{p_1}^{-1}\left(\int_{0}^{1}g_1(\tau)d\tau\right),$$
$$\frac{1}{m_2}=
\max\left\{\int_{0}^{\frac{1}{2}}\varphi_{p_2}^{-1}\left(\int_{s}^{\frac{1}{2}}g_2(\tau)d\tau\right)ds+h_{22}\varphi_{p_2}^{-1}\left(\int_{0}^{\frac{1}{2}}g_2(\tau)d\tau\right),\int_{\frac{1}{2}}^{1}\varphi_{p_2}^{-1}\left(\int_{\frac{1}{2}}^{s}g_2(\tau)d\tau\right)ds\,\right\}.$$

Then $i_{K}(T,K_{\rho_1,\rho_2})=1$.
\end{lem}

\begin{proof}
We show that $\lambda (u,v)\neq T(u,v)$ for every $(u,v)\in \partial K_{\rho_1,\rho_2 }$ and for every $\lambda \geq 1$; this ensures that the index is 1 on $K_{\rho_1,\rho_2 }$. In fact, if this does not happen, there exist $\lambda \geq 1$ and $(u,v)\in \partial K_{\rho_1,\rho_2 }$ such that $\lambda (u,v)=T(u,v)$.\\
Firstly we assume  that $\| u\| _{\infty }=\rho_1 $ and $\| v\| _{\infty}\leq \rho_2 $. \\
Then we have
\begin{align*}
\lambda u(t)&= \int_{t}^{1} \varphi_{p_1}^{-1}\left(\int_{0}^{s}
g_1(\tau)f_1(\tau,u(\tau),v(\tau))d\tau\right)ds+B_1\left(\varphi_{p_1}^{-1}\left(\int_{0}^{1}g_1(\tau)
f_1(\tau,u(\tau),v(\tau))d\tau
\right)\right)\\
&\le \int_{t}^{1} \varphi_{p_1}^{-1}\left(\int_{0}^{s}
g_1(\tau)f_1(\tau,u(\tau),v(\tau))d\tau\right)ds+h_{12}\varphi_{p_1}^{-1}\left(\int_{0}^{1}g_1(\tau)
f_1(\tau,u(\tau),v(\tau))d\tau
\right)\\
&=\rho_1  \left(\int_{t}^{1} \varphi_{p_1}^{-1}(\int_{0}^{s}
g_1(\tau)\frac{f_1(\tau,u(\tau),v(\tau))}{\rho^{p_1-1}_1}d\tau)ds+h_{12}\varphi_{p_1}^{-1}(\int_{0}^{1}g_1(\tau)
\frac{f_1(\tau,u(\tau),v(\tau))}{\rho^{p_1-1}_1}d\tau
)\right).\\
\end{align*}
Taking  $t=0$ gives
\begin{eqnarray*}
\lambda u(0)&=&\lambda {\rho_1}\leq \rho_1  \left(\int_{0}^{1}
\varphi_{p_1}^{-1}\left(\int_{0}^{s}
g_1(\tau)f_1^{\rho_1,\rho_2}d\tau\right)\,ds+
h_{12}\varphi_{p_1}^{-1}\left(\int_{0}^{1}g_1(\tau)f_1^{\rho_1,\rho_2}d\tau\right)\right)\\
&=&  \rho_1
\varphi_{p_1}^{-1}(f_1^{\rho_1,\rho_2})\left(\int_{0}^{1}
\varphi_{p_1}^{-1}\left( \int_{0}^{s}
g_1(\tau)d\tau\right)\,ds+h_{12}\varphi_{p_1}^{-1}\left(\int_{0}^{1}g_1(\tau)d\tau\right)\right)\\
 &=&{\rho_1}\frac{1}{m_1} \varphi_{p_1}^{-1}\left( f_1^{\rho_1,\rho_2}\right).
\end{eqnarray*}
Using the hypothesis \eqref{eqmestt} and the strictly monotonicity of $\varphi_{p_1}^{-1}$ we obtain $\lambda \rho_1 <\rho_1 .$ This
contradicts the fact that $\lambda \geq 1$ and proves the result.\\
Now we assume  $\|v\|_{\infty}=\rho_2$ and $\| u\| _{\infty}\leq \rho_1 $.\\
 Then we have
 $$
 \lambda \rho_2= \|T_2(u,v)\|_\infty=T_2(u,v)(\sigma_{u,v}).
 $$
 If $\sigma_{u,v}\le \displaystyle{\frac{1}{2}}$, we have
\begin{align*}
&\lambda \rho_2= \|T_2(u,v)\|_\infty=T_2(u,v)(\sigma_{u,v})\\&=
\int_{0}^{\sigma_{u,v}}\varphi_{p_2}^{-1}\left(\int_{s}^{\sigma_{u,v}}g_2(\tau)f_2(\tau,u(\tau),v(\tau))d\tau\right)ds
+B_2\left(\varphi_{p_2}^{-1}\left(\int_{0}^{\sigma_{u,v}}g_2(\tau)f_2(\tau,u(\tau),v(\tau))d\tau\right)\right) \\
&\leq \int_{0}^{\frac{1}{2}}
\varphi_{p_2}^{-1}\left(\int_{s}^{\frac{1}{2}}
g_2(\tau)f_2(\tau,u(\tau),v(\tau))d\tau\right)ds+h_{22}\,\varphi_{p_2}^{-1}\left(\int_{0}^{\sigma_{u,v}}g_2(\tau)f_2(\tau,u(\tau),v(\tau))d\tau\right)\\
&\leq \int_{0}^{\frac{1}{2}}
\varphi_{p_2}^{-1}\left(\int_{s}^{\frac{1}{2}}
g_2(\tau)f_2(\tau,u(\tau),v(\tau))d\tau\right)ds+h_{22}\,\varphi_{p_2}^{-1}\left(\int_{0}^{\frac{1}{2}}g_2(\tau)f_2(\tau,u(\tau),v(\tau))d\tau\right)\\
&=\rho_2 \int_{0}^{\frac{1}{2}}
\varphi_{p_2}^{-1}\left(\int_{s}^{\frac{1}{2}}g_2(\tau)\frac{f_2(\tau,u(\tau),v(\tau))}{\rho_2^{p_2-1}}d\tau\right)ds+
h_{22}\,\varphi_{p_2}^{-1}\left(\int_{0}^{\frac{1}{2}}g_2(\tau)\frac{f_2(\tau,u(\tau),v(\tau))}{\rho_2^{p_2-1}}d\tau\right)\\
&\leq \rho_2
\varphi_{p_2}^{-1}(f_2^{\rho_1,\rho_2})\int_{0}^{\frac{1}{2}}\varphi_{p_2}^{-1}\left(\int_{s}^{\frac{1}{2}}g_2(\tau)d\tau\right)ds+
h_{22}\,\varphi_{p_2}^{-1}\left(\int_{0}^{\frac{1}{2}}g_2(\tau)d\tau\right).
\end{align*}
 If $\sigma_{u,v}>\displaystyle{\frac{1}{2}}$, we have
\begin{align*}
&\lambda \rho_2=\|T_2(u,v)\|_\infty=T_2(u,v)(\sigma_{u,v})\\&=
\int_{\sigma_{u,v}}^{1}\varphi_{p_2}^{-1}\left(\int_{\sigma_{u,v}}^{s}g_2(\tau)f_2(\tau,u(\tau),v(\tau))d\tau\right)\,ds\leq \int_{\frac{1}{2}}^{1}\varphi_{p_2}^{-1}\left(\int_{\frac{1}{2}}^{s}g_2(\tau)f_2(\tau,u(\tau),v(\tau))d\tau\right)\,ds\\
&=\rho_2 \int_{\frac{1}{2}}^{1}\varphi_{p_2}^{-1}\left(\int_{\frac{1}{2}}^{s}g_2(\tau)\frac{f_2(\tau,u(\tau),v(\tau))}{\rho_2^{p_2-1}}d\tau\right)\,ds\leq\rho_2 \varphi_{p_2}^{-1}(f_2^{\rho_1,\rho_2})\int_{\frac{1}{2}}^{1}\varphi_{p_2}^{-1}\left(\int_{\frac{1}{2}}^{s}g_2(\tau)d\tau\right)\,ds.\\
\end{align*}
Then, in both cases, we have
\begin{align*}
 &\lambda \rho_2= \|T_2(u,v)\|_\infty=T_2(u,v)(\sigma_{u,v})\leq\rho_2 \varphi_{p_2}^{-1}(f_2^{\rho_1,\rho_2})\times\\
 &\max\Bigg\{\int_{0}^{\frac{1}{2}}\varphi_{p_2}^{-1}\left(\int_{s}^{\frac{1}{2}}g_2(\tau)d\tau\right)\,ds+h_{22}\,\varphi_{p_2}^{-1}\left(\int_{0}^{\frac{1}{2}}g_2(\tau)d\tau\right),\int_{\frac{1}{2}}^{1}
\varphi_{p_2}^{-1}\left(\int_{\frac{1}{2}}^{s}g_2(\tau)d\tau\right)\,ds\,\Bigg\}\\
&=\rho_2 \varphi_{p_2}^{-1}(f_2^{\rho_1,\rho_2})\frac{1}{m_2}.
\end{align*}
Using the hypothesis \eqref{eqmestt} and the strictly monotonicity of $\varphi_{p_2}^{-1}$ we obtain $\lambda \rho_2 <\rho_2 .$ This
contradicts the fact that $\lambda \geq 1$ and proves the result.
\end{proof}

We give a first Lemma that shows that the index is 0 on a set $V_{\rho_1,\rho_2 }$.
\begin{lem}\label{idx0b1}
Assume that:\ \\
$(\mathrm{I}^{0}_{\rho_1,\rho_2})$  there exist $\rho_1,\rho_2>0$ such that for every $i=1,2$
\begin{equation}\label{eqMest}
f_{i,(\rho_1, \rho_2)} > \varphi_{p_i}(M_i),
\end{equation}{}
where
$$f_{1,({\rho_1,\rho_2 })}= \inf \Bigl\{ \frac{f_1(t,u,v)}{ \rho_1^{p_1-1}}:\; (t,u,v)\in [0,b_1]\times[c_1\rho_1,\rho_1]\times[0, \rho_2]\Bigr\},$$
$$f_{2,({\rho_1,\rho_2 })}= \inf \Bigl\{ \frac{f_2(t,u,v)}{ \rho_2^{p_2-1}}:\; (t,u,v)\in [a_2,b_2]\times[0,\rho_1]\times[c_2\rho_2, \rho_2]\Bigr\},$$
$$\frac{1}{M_1}= \int_{0}^{b_1} \varphi_{p_1}^{-1}\left(\int_{0}^{s} g_1(\tau)\,d\tau\right)ds+h_{11}\,\varphi_{p_1}^{-1}\left(\int_{0}^{b_1}g_1(\tau)d\tau\right),$$
and
 $$\frac{1}{M_2}=\frac{1}{2}\min _{a_2\leq \nu\leq b_2} \left\{\int_{a_2}^{\nu}\varphi_{p_2}^{-1}(\int_{s}^{\nu}
g_2(\tau)d\tau)ds
+\int_{\nu}^{b_2}\varphi_{p_2}^{-1}(\int_{\nu}^{s}
g_2(\tau)d\tau)ds+h_{21}\varphi_{p_2}^{-1}(\int_{a_2}^{\nu}g_2(\tau)d\tau)\right\}.$$
Then $i_{K}(T,V_{\rho_1,\rho_2})=0$.
\end{lem}
\begin{proof}
Let $e(t)\equiv 1$ for $t\in [0,1]$. Then $(e,e)\in K$. We prove that
\begin{equation*}
(u,v)\ne T(u,v)+\lambda (e,e)\quad\text{for } (u,v)\in \partial V_{\rho_1,\rho_2 }\quad\text{and } \lambda \geq 0.
\end{equation*}
In fact, if this does not happen, there exist $(u,v)\in \partial V_{\rho_1,\rho_2 }$ and $\lambda \geq 0$ such that
$(u,v)=T(u,v)+\lambda (e,e)$. We examine the two cases:\\

Case $(1)$:   $ c_1\rho_1 \le u(t) \le \rho_1$  for  $t\in [0,b_1]$  and  $0\leq v(t)\leq \rho_2$ for  $t\in [0,1]$.\\
 Thus for $t\in [0,b_1]$, we have
\begin{align*}
&\rho_1\geq u(t)\ \\
&= \int_{t}^{1}\varphi_{p_1}^{-1}\left(\int_0^s g_1(\tau)f_1(\tau,u(\tau),v(\tau))d\tau\right)ds+B_1\left(\varphi_{p_1}^{-1}\left(\int_{0}^{1}g_1(\tau)
f_1(\tau,u(\tau),v(\tau))d\tau\right)\right) +\lambda\ \\
 &\geq \int_{t}^{b_1}\varphi_{p_1}^{-1}\left(\int_{0}^{s}
g_1(\tau)f_1(\tau,u(\tau),v(\tau))d\tau\right)\,ds+h_{11}\,\varphi_{p_1}^{-1}\left(\int_{0}^{1}g_1(\tau)
f_1(\tau,u(\tau),v(\tau))d\tau
\right) + \lambda\ \\
&\geq \int_{t}^{b_1}\varphi_{p_1}^{-1}\left(\int_{0}^{s}
g_1(\tau)f_1(\tau,u(\tau),v(\tau))d\tau\right)\,ds+h_{11}\,\varphi_{p_1}^{-1}\left(\int_{0}^{b_1}g_1(\tau)
f_1(\tau,u(\tau),v(\tau))d\tau
\right) + \lambda\ \\
&=\rho_1\int_{t}^{b_1}\varphi_{p_1}^{-1}\left(\int_{0} ^{s}
g_1(\tau)\frac{f_1(\tau,u(\tau),v(\tau))}{\rho_1^{p_1-1}}d\tau\right)ds+\rho_1 h_{11}\varphi_{p_1}^{-1}\left(\int_{0}^{b_1}g_1(\tau)
\frac{f_1(\tau,u(\tau),v(\tau))}{\rho_1^{p_1-1}}d\tau \right)
+ \lambda.
\end{align*}
For $t=0$ we obtain
\begin{align*}
\rho_1&\geq\rho_1\varphi_{p_1}^{-1}(f_1,_{(\rho_1,\rho_2)})\left(\int_{0}^{b_1}\varphi_{p_1}^{-1}\left(\int_{0}^{s}
g_1(\tau)\,d\tau\right)\,ds+h_{11}\,\varphi_{p_1}^{-1}\left(\int_{0}^{b_1}g_1(\tau)d\tau\right)\right)
+ \lambda\\
&>{\rho_1}\varphi_{p_1}^{-1}\left(f_{1,(\rho_1,\rho_2)}\right)\frac{1}{M_{1}}+{\lambda}.
\end{align*}
Using the hypothesis \eqref{eqMest} we obtain $\rho_1>\rho_1 +\lambda $, a contradiction.\\
 Case $(2)$:    $ 0\leq u(t)\leq \rho_1$ for  $t\in [0,1]$   and    $c_2\rho_2 \le v(t) \le \rho_2$.\\
 We  distinguish three cases:\ \\
Case (i)   $0<\sigma_{u,v}\leq a_2$.\ \\
 Therefore we get
\begin{align*}
\rho_2&\geq v(\sigma_{u,v})=T_2(u,v)(\sigma_{u,v})+\lambda=\int_{\sigma_{u,v}}^{1}\varphi_{p_2}^{-1}\left(\int_{\sigma_{u,v}}^{s}
g_2(\tau)f_2(\tau,u(\tau),v(\tau))d\tau\right)\,ds+\lambda\\
&\geq\int_{a_2}^{b_2}\varphi_{p_2}^{-1}\left(\int_{a_2}^{s} g_2(\tau)f_2(\tau,u(\tau),v(\tau))d\tau\right)\,ds+\lambda\\
&=\rho_2 \int_{a_2}^{b_2}\varphi_{p_2}^{-1}\left(\int_{a_2}^{s} g_2(\tau)\frac{f_2(\tau,u(\tau),v(\tau))}{\rho_2^{p_2-1}}\,d\tau\right)ds+\lambda\\
&\geq\rho_2\varphi_{p_2}^{-1}(f_2,_{(\rho_1,\rho_2)})\left(\int_{a_2}^{b_2}\varphi_{p_2}^{-1}\left(\int_{a_2}^{s}
g_2(\tau)\,d\tau\right)ds\right)+\lambda
\geq\rho_2\varphi_{p_2}^{-1}(f_2,_{(\rho_1,\rho_2)})\frac{1}{M_2}+\lambda.
\end{align*}
Using the hypothesis \eqref{eqMest} we obtain $\rho_2>\rho_2 +\lambda $, a contradiction.\ \\
Case (ii) $\sigma_{u,v} \geq b_2$.
\begin{eqnarray*}
&&\rho_2\geq v(\sigma_{u,v})=T_2(u,v)(\sigma_{u,v})+\lambda
=\int_{0}^{\sigma_{u,v}} \varphi_{p_2}^{-1}\left(\int_{s}^{\sigma_{u,v}} g_2(\tau)f_2(\tau,u(\tau),v(\tau))d\tau\right)ds\\&&+B_2\left(\varphi_{p_2}^{-1}\left(\int_{0}^{\sigma_{u,v}}g_2(\tau)f_2(\tau,u(\tau),v(\tau))d\tau\right)\right)  +\lambda\\
&&\geq \int_{a_2}^{b_2} \varphi_{p_2}^{-1}\left(\int_{s}^{b_2}
g_2(\tau)f_2(\tau,u(\tau),v(\tau))d\tau\right)ds+
h_{21}\,\varphi_{p_2}^{-1}\left(\int_{a_2}^{b_2}g_2(\tau)f_2(\tau,u(\tau),v(\tau))d\tau\right) +\lambda\\
&&=\rho_2 \int_{a_2}^{b_2}\varphi_{p_2}^{-1}\left(\int_{s}^{b_2}
g_2(\tau)\frac{f_2(\tau,u(\tau),v(\tau))}{\rho_2^{p_1-1}}d\tau\right)
ds\\&&+
\rho_2 h_{21}\,\varphi_{p_2}^{-1}\left(\int_{a_2}^{b_2}g_2(\tau)\frac{f_2(\tau,u(\tau),v(\tau))}{\rho_2^{p_2-1}}\,d\tau\right) +\lambda\\
&&\geq\rho_2\varphi_{p_2}^{-1}(f_2,_{(\rho_1,\rho_2)})\left(\int_{a_2}^{b_2}\varphi_{p_2}^{-1}\left(\int_{s}^{b_2}
g_2(\tau)\,d\tau\right)ds+h_{21}\,\varphi_{p_2}^{-1}\left(\int_{a_2}^{b_2}g_2(\tau)d\tau\right)\right)\,
+\lambda\\
&&\geq\rho_2\varphi_{p_2}^{-1}(f_2,_{(\rho_1,\rho_2)})\frac{1}{M_2}+\lambda.
\end{eqnarray*}
Using the hypothesis \eqref{eqMest} we obtain $\rho_2>\rho_2 +\lambda $, a contradiction.\\
Case (iii)  $a_2<\sigma_{u,v}<b_2$.
\begin{align*}
2\rho_2&\geq 2v(\sigma_{u,v})=2\lambda+2T_2(u,v)(\sigma_{u,v})= 2\lambda+\int_{0}^{\sigma_{u,v}} \varphi_{p_2}^{-1}\left(\int_{s}^{\sigma_{u,v}} g_2(\tau)f_2(\tau,u(\tau),v(\tau))
d\tau\right)\,ds\\
&+B_2\left(\varphi_{p_2}^{-1}\left(\int_{0}^{\sigma_{u,v}}g_2(\tau)f_2(\tau,u(\tau),v(\tau))\,d\tau\right)\right) +\int_{\sigma_{u,v}}^{1} \varphi_{p_2}^{-1}\left(\int_{\sigma_{u,v}}^{s} g_2(\tau)f_2(\tau,u(\tau),v(\tau))d\tau\right)\,ds\\
&\geq 2\lambda+\int_{a_2}^{\sigma_{u,v}}
\varphi_{p_2}^{-1}\left(\int_{s}^{\sigma_{u,v}}
g_2(\tau)f_2(\tau,u(\tau),v(\tau))d\tau\right)\,ds\\
&+h_{21}\left(\varphi_{p_2}^{-1}\left(\int_{a_2}^{\sigma_{u,v}}g_2(\tau)f_2(\tau,u(\tau),v(\tau))\,d\tau\right)\right)+
\int_{\sigma_{u,v}}^{b_2} \varphi_{p_2}^{-1}\left(\int_{\sigma_{u,v}}^{s} g_2(\tau)f_2(\tau,u(\tau),v(\tau))d\tau\right)\,ds\\
&=2\lambda+\rho_2 \Bigg[\int_{a_2}^{\sigma_{u,v}}
\varphi_{p_2}^{-1}\left(\int_{s}^{\sigma_{u,v}}
g_2(\tau)\frac{f_2(\tau,u(\tau),v(\tau))}{\rho_2^{p_2-1}}d\tau\right)\,ds\\
&+h_{21}\,\varphi_{p_2}^{-1}\left(\int_{a_2}^{\sigma_{u,v}}g_2(\tau)\frac{f_2(\tau,u(\tau),v(\tau))}{\rho_2^{p_2-1}}\,d\tau\right)
+ \int_{\sigma_{u,v}}^{b_2} \varphi_{p_2}^{-1}\left(\int_{\sigma_{u,v}}^{s}g_2(\tau)\frac{f_2(\tau,u(\tau),v(\tau))}{\rho_2^{p_2-1}}d\tau\right)\,ds\Bigg]\\
&\geq
2\lambda+\rho_2\varphi_{p_2}^{-1}(f_2,_{(\rho_1,\rho_2)})\Bigg[\int_{a_2}^{\sigma_{u,v}}
\varphi_{p_2}^{-1}\left(\int_{s}^{\sigma_{u,v}}
g_2(\tau)d\tau\right)\,ds \\
&+h_{21}\,\varphi_{p_2}^{-1}\left(\int_{a_2}^{\sigma_{u,v}}g_2(\tau)\,d\tau\right)
+ \int_{\sigma_{u,v}}^{b_2} \varphi_{p_2}^{-1}\left(\int_{\sigma_{u,v}}^{s}g_2(\tau)d\tau\right)\,ds\Bigg]\\
&\geq 2\lambda+
2\rho_2\varphi_{p_2}^{-1}(f_2,_{(\rho_1,\rho_2)})\frac{1}{M_2}.
\end{align*}
Using the hypothesis \eqref{eqMest} we obtain $\rho_2>\lambda+\rho_2 $, a contradiction.\\
\end{proof}
\begin{rem}
We point out that a stronger, but easier to check, hypothesis  than \eqref{eqMest} is 
\begin{equation*}
f_{i,(\rho_1, \rho_2)} > \varphi_{p_i}(\tilde{M}_i),
\end{equation*}{}
where
$$\frac{1}{\tilde{M}_1}= \int_{0}^{b_1} \varphi_{p_1}^{-1}\left(\int_{0}^{s} g_1(\tau)\,d\tau\right)ds$$
and
 $$\frac{1}{\tilde{M}_2}=\frac{1}{2}\min _{a_2\leq \nu\leq b_2} \left\{\int_{a_2}^{\nu}\varphi_{p_2}^{-1}(\int_{s}^{\nu}
g_2(\tau)d\tau)ds
+\int_{\nu}^{b_2}\varphi_{p_2}^{-1}(\int_{\nu}^{s}
g_2(\tau)d\tau)ds\right\}.$$
\end{rem}

In the following Lemma we exploit an idea that was used in \cite{gipp-nonlin, gipp-nodea} and we provide a result of index 0 controlling the growth of just one nonlinearity $f_i$, at the cost of having to deal with a larger domain. Nonlinearities with different growths were considered for examples in~\cite{precup1, precup2, ya1}.
\begin{lem}\label{idx0b3}
Assume that
\begin{enumerate}
\item[$(\mathrm{I}^{0}_{\rho_1,\rho_2})^{\star}$] there exist $\rho_1,\rho_2>0$ such that for some $i\in\{1,2\}$ we have
\begin{equation}\label{eqMest1}
f^*_{i,(\rho_1, \rho_2)}>\varphi_{p_i}(M_i),
\end{equation}{}
\end{enumerate}
where
\begin{equation*}
f^*_{i,(\rho_1,{\rho_2})}=\inf \Bigl\{ \frac{f_i(t,u,v)}{ \rho_i^{p_i-1}}:\; (t,u,v)\in [a_i,b_i]\times[0,\rho_1]\times[0, \rho_2]\Bigr\}.
\end{equation*}
Then $i_{K}(T,V_{\rho_1,\rho_2})=0$.
\end{lem}
\begin{proof}
Suppose that the condition \eqref{eqMest1} holds for $i=1$. Let  $(u,v)\in \partial V_{\rho_1,\rho_2 }$ and $\lambda \geq 0$ such that
$(u,v)=T(u,v)+\lambda (e,e)$. Thus we proceed as in the proof of Lemma \ref {idx0b1}.
\end{proof}

The proof of the next result regarding the existence of at least one, two or three positive solutions follows by the properties of fixed point index and is omitted. It is possible to state results for four or more positive solutions, in a similar way as in~\cite{kljdeds},  by expanding the lists in conditions $(S_{5}),(S_{6})$.
\begin{thm}\label{mult-sys}
The system \eqref{syst} has at least one positive solution in $K$ if one of the following conditions holds.
\begin{enumerate}
\item[$(S_{1})$]  For $i=1,2$ there exist $\rho _{i},r _{i}\in (0,\infty )$ with $\rho_{i}<r _{i}$ such that $(\mathrm{I}_{\rho _{1},\rho_2}^{0})\;\;[\text{or}\;(\mathrm{I}_{\rho _{1},\rho_2}^{0})^{\star }]$, $(\mathrm{I}_{r _{1},r_2}^{1})$ hold.
\item[$(S_{2})$] For $i=1,2$ there exist $\rho _{i},r _{i}\in (0,\infty )$ with $\rho_{i}<c_i r _{i}$ such that $(\mathrm{I}_{\rho _{1},\rho_2}^{1}),\;\;(\mathrm{I}_{r _{1},r_2}^{0})$ hold.
\end{enumerate}
The system \eqref{syst} has at least two positive solutions in $K$ if one of the following conditions holds.
\begin{enumerate}
\item[$(S_{3})$] For $i=1,2$ there exist $\rho _{i},r _{i},s_i\in (0,\infty )$ with $\rho _{i}<r_i <c_i s _{i}$ such that $(\mathrm{I}_{\rho_{1},\rho_2}^{0})$,
$[\text{or}\;(\mathrm{I}_{\rho _{1},\rho_2}^{0})^{\star }],\;(\mathrm{I}_{r _{1},r_2}^{1})$ $\text{and}\;\;(\mathrm{I}_{s _{1},s_2}^{0})$ hold.
\item[$(S_{4})$] For $i=1,2$ there exist $\rho _{i},r _{i},s_i\in (0,\infty )$ with $\rho _{i}<c_i r _{i}$ and $r _{i}<s _{i}$ such that
$(\mathrm{I}_{\rho _{1},\rho_2}^{1})$, $(\mathrm{I}_{r _{1},r_2}^{0})$ $\text{and}\;\;(\mathrm{I}_{s _{1},s_2}^{1})$ hold.
\end{enumerate}
The system \eqref{syst} has at least three positive solutions in $K$ if one of the following conditions holds.
\begin{enumerate}
\item[$(S_{5})$] For $i=1,2$ there exist $\rho _{i},r _{i},s_i,\delta_i\in (0,\infty )$ with $\rho _{i}<r _{i}<c_i s _{i}$ and $s _{i}<\delta_{i}$ such that
$(\mathrm{I}_{\rho _{1},\rho_2}^{0})\;\;[\text{or}\;(\mathrm{I}_{\rho _{1},\rho_2}^{0})^{\star }],$ $(\mathrm{I}_{r _{1},r_2}^{1}),
\;\;(\mathrm{I}_{s_1,s_2}^{0})\;\;\text{and}\;\;(\mathrm{I}_{\delta _{1},\delta_2}^{1})$ hold.
\item[$(S_{6})$] For $i=1,2$ there exist $\rho _{i},r _{i},s_i,\delta_i\in (0,\infty )$ with $\rho _{i}<c_i r _{i}$ and $r _{i}<s _{i}<c_i \delta _{i}$
such that $(\mathrm{I}_{\rho _{1},\rho_2}^{1}),\;\;(\mathrm{I}_{r_{1},r_2}^{0}),\;\;(\mathrm{I}_{s _{1},s_2}^{1})$ $\text{and}
\;\;(\mathrm{I}_{\delta _{1},\delta_2}^{0})$ hold.
\end{enumerate}
\end{thm}
\section{Non-existence results}
We now provide some  non-existence results for system \eqref{syst}.\\

\begin{thm}
Assume that one of the following conditions holds.
\begin{enumerate}
\item For  $i=1,2$,
\begin{equation}\label{cond1}
f_i(t,u_1,u_2)< \varphi _{p_i}(m_i u_i)\ \text{for every}\ t\in [0,1] \text{        and        } u_i>0.
\end{equation}
\item For  $i=1,2$,
\begin{equation}\label{cond2}
f_i(t,u_1,u_2)> \varphi _{p_i}\left(\frac{M_i}{c_i} u_i\right) \text{for every}\ t\in [a_i,b_i] \text{        and        }  u_i>0.
\end{equation}
\item There exists $k\in\{1,2\}$ such that \eqref{cond1} is verified for $f_k$ and for  $j\neq k$  condition \eqref{cond2}  is verified for $f_j$.
\end{enumerate}
Then there is no positive solution of the system \eqref{syst} in $K$.
\end{thm}
\begin{proof}
$(1)$ Assume, on the contrary, that there exists $(u,v)\in K$ such that $(u,v)=T(u,v)$ and $(u,v)\neq (0,0)$.
We distinguish two cases.
\begin{itemize}
\item Let be $\|u\|_\infty \neq 0$.
Then we have
\begin{align*}
u(t)&= \int_{t}^{1} \varphi_{p_1}^{-1}\left(\int_{0}^{s}
g_1(\tau)f_1(\tau,u(\tau),v(\tau))d\tau\right)ds+B_1\left(\varphi_{p_1}^{-1}\left(\int_{0}^{1}g_1(\tau)
f_1(\tau,u(\tau),v(\tau))d\tau
\right)\right) \\
&< m_1  \int_{t}^{1} \varphi_{p_1}^{-1}\left(\int_{0}^{s} g_1(\tau)\varphi_{p_1}(u(\tau))d\tau\right)ds+
m_1 h_{12} \varphi_{p_1}^{-1}\left(\int_{0}^{1} g_1(\tau)\varphi_{p_1}(u(\tau))d\tau\right)\\
&\leq m_1  \|u\|_\infty \left(\int_{t}^{1}
\varphi_{p_1}^{-1}\left(\int_{0}^{s}
g_1(\tau)d\tau\right)ds+h_{12}
\varphi_{p_1}^{-1}\left(\int_{0}^{1} g_1(\tau)d\tau\right)\right).
\end{align*}
Taking  $t=0$ gives
\begin{align*}
\|u\|_\infty &= u(0)< m_1  \|u\|_\infty \left(\int_{0}^{1}
\varphi_{p_1}^{-1}\left(\int_{0}^{s}
g_1(\tau)d\tau\right)ds+h_{12}
\varphi_{p_1}^{-1}\left(\int_{0}^{1}
g_1(\tau)d\tau\right)\right)\\&= m_1  \|u\|_\infty \frac{1}{m_1},
\end{align*}
a contradiction.
\item Let be  $\|v\|_{\infty} \neq 0$.\\
 Reasoning as in Lemma \ref {ind1b} we distinguish the  cases $\sigma_{u,v} \leq 1/2$ and $\sigma_{u,v} >1/2$.\\ In the first case we have
\begin{align*}
\|v\|_{\infty} &= \|T_2(u,v)\|_\infty=T_2(u,v)(\sigma_{u,v})\\
&= \int_{0}^{\sigma_{u,v}} \varphi_{p_2}^{-1}\left(\int_{s}^{\sigma_{u,v}} g_2(\tau)f_2(\tau,u(\tau),v(\tau))d\tau\right)\,ds\\&+B_2\left(\varphi_{p_2}^{-1}\left(\int_{0}^{\sigma_{u,v}}g_2(\tau)f_2(\tau,u(\tau),v(\tau))\,d\tau\right)\right)\\
&< m_2  \|v\|_\infty  \left(\int_{0}^{\sigma_{u,v}}
\varphi_{p_2}^{-1}\left(\int_{s}^{\sigma_{u,v}}
g_2(\tau)d\tau\right)\,ds+ h_{22}
\varphi_{p_2}^{-1}\left(\int_{0}^{\sigma_{u,v}}g_2(\tau)\,d\tau\right)\right)\\
&\le m_2  \|v\|_\infty  \left(\int_{0}^{\frac{1}{2}}
\varphi_{p_2}^{-1}\left(\int_{s}^{\frac{1}{2}}
g_2(\tau)d\tau\right)\,ds+ h_{22}
\varphi_{p_2}^{-1}\left(\int_{0}^{\frac{1}{2}}g_2(\tau)\,d\tau\right)\right)\le
m_2  \|v\|_\infty\frac{1}{m_2}\,,
\end{align*}
a contradiction. \\The proof is similar in the last case
$\sigma_{u,v} >1/2$.
\end{itemize}
$(2)$ Assume, on the contrary, that there exists $(u,v)\in K$ such that $(u,v)=T(u,v)$ and$(u,v)\neq (0,0)$.
 We distinguish two cases
\begin{itemize}
\item Let be $\|u\|_\infty \neq 0$.
 Then, for $t\in [a_1,b_1]= [0,b_1]$, we have
\begin{align*}
u(t)&= \int_{t}^{1}\varphi_{p_1}^{-1}\left(\int_0^s
g_1(\tau)f_1(\tau,u(\tau),v(\tau))\,d\tau\right)\,ds
+B_1\left(\varphi_{p_1}^{-1}\left(\int_{0}^{1}g_1(\tau)
f_1(\tau,u(\tau),v(\tau))\,d\tau
\right)\right)\\
&\geq \int_{t}^{b_1}\varphi_{p_1}^{-1}\left(\int_{0}^{s}
g_1(\tau)f_1(\tau,u(\tau),v(\tau))\,d\tau\right)\,ds
+h_{11}\varphi_{p_1}^{-1}\left(\int_{0}^{1}g_1(\tau)
f_1(\tau,u(\tau),v(\tau))\,d\tau
\right)\\
&\geq \int_{t}^{b_1}\varphi_{p_1}^{-1}\left(\int_{0}^{s}
g_1(\tau)f_1(\tau,u(\tau),v(\tau))\,d\tau\right)\,ds
+h_{11}\varphi_{p_1}^{-1}\left(\int_{0}^{b_1}g_1(\tau)
f_1(\tau,u(\tau),v(\tau))\,d\tau
\right)\\
&>\frac{M_1}{c_1}\left(\int_{t}^{b_1}\varphi_{p_1}^{-1}\left(\int_{0}
^{s} g_1(\tau)\varphi_{p_1}(u(\tau))\,d\tau\right)\,ds+
h_{11}\varphi_{p_1}^{-1}\left(\int_{0}^{b_1}g_1(\tau)\varphi_{p_1}(u(\tau))d\tau\right)\right)\\
&>\frac{M_1}{c_1}\left(\int_{t}^{b_1}\varphi_{p_1}^{-1}\left(\int_{0}
^{s}
g_1(\tau)\varphi_{p_1}(c_1\|u\|_{\infty})\,d\tau\right)\,ds+h_{11}\varphi_{p_1}^{-1}\left(\int_{0}^{b_1}g_1(\tau)
\varphi_{p_1}(c_1\|u\|_{\infty})\,d\tau \right)\right)\\
&=\frac{M_1}{c_1}c_1\|u\|_{\infty}\left(\int_{t}^{b_1}\varphi_{p_1}^{-1}\left(\int_{0}
^{s} g_1(\tau)\,d\tau\right)\,ds+
h_{11}\varphi_{p_1}^{-1}\left(\int_{0}^{b_1}g_1(\tau)d\tau
\right)\right).
\end{align*}
For $t=0$ we obtain
\begin{equation*}
u(0)= \|u\|_{\infty}>M_1 \|u\|_{\infty}\frac{1}{M_{1}},
\end{equation*}
 a contradiction.
\item Let be  $\|v\|_{\infty} \neq 0$. We examine the case $\sigma_{u,v} \geq b_2$.
We have
\begin{align*}
\|v\|_{\infty}&=v(\sigma_{u,v})=T_2(u,v)(\sigma_{u,v})
= \int_{0}^{\sigma_{u,v}} \varphi_{p_2}^{-1}\left(\int_{s}^{\sigma_{u,v}} g_2(\tau)f_2(\tau,u(\tau),v(\tau))d\tau\right)\,ds\\&+B_2\left(\varphi_{p_2}^{-1}\left(\int_{0}^{\sigma_{u,v}}g_2(\tau)f_2(\tau,u(\tau),v(\tau))\,d\tau\right)\right) \\
&\geq \int_{a_2}^{b_2} \varphi_{p_2}^{-1}\left(\int_{s}^{b_2}
g_2(\tau)f_2(\tau,u(\tau),v(\tau))d\tau\right)\,ds +
h_{21}\,\varphi_{p_2}^{-1}\left(\int_{a_2}^{b_2}g_2(\tau)f_2(\tau,u(\tau),v(\tau))\,d\tau\right) \\
&> \frac{M_2}{c_2}c_2\|v\|_{\infty}\Bigg(\int_{a_2}^{b_2}
\varphi_{p_2}^{-1}\left(\int_{s}^{b_2} g_2(\tau)\,d\tau\right)ds
+h_{21}\,\varphi_{p_2}^{-1}\left(\int_{a_2}^{b_2}g_2(\tau)d\tau\right)\Bigg)
\geq M_2\|v\|_{\infty}\frac{1}{M_2},
\end{align*}
 a contradiction. By similar proofs, the cases $0<\sigma_{u,v}\leq a_2$ and $a_2<\sigma_{u,v}< b_2$ can be examined.
\end{itemize}
$(3)$ Assume, on the contrary, that there exists $(u,v)\in K$ such that $(u,v)=T(u,v)$ and $(u,v)\neq (0,0)$. If  $\|u\|_\infty \neq 0$
then the function $f_1$ satisfies either \eqref{cond1} or \eqref{cond2} and the proof follows as in the previous cases.  If  $\|v\|_\infty \neq 0$
then the function $f_2$ satisfies either \eqref{cond1} or \eqref{cond2} and the proof follows as  previous cases.
\end{proof}
\section{An example}
We illustrate in the following example that all the constants that occur in the Theorem~\ref{mult-sys} can be computed.\\
Consider the system
\begin{equation}\label{lap2}
\begin{array}{c}
(\varphi_{p_1}(u^{\prime}))'(t)+g_1(t)f_{1}(t, u(t),v(t))=0,\ t \in (0,1), \\
(\varphi_{p_2}(v^{\prime}))'(t)+g_2(t)f_{2}(t, u(t),v(t))=0,\ t \in (0,1), \\
\end{array}%
\end{equation}
subject to  boundary conditions
\begin{equation}\label{bc2}
u'(0)=
0\,,\,\,u(1)+B_1\left(u'(1)\right)=0,\,\,\,\,v(0)=B_2\left(v'(0)\right)\,,\,\,v(1)=0,
\end{equation}
where $B_1$ and $B_2$ are defined by:
$$B_1(w)=\begin{cases} w, \,\,\,\,\,\,w\le 0, \cr
\frac{w}{2}, \,\,\,\, 0\leq w \leq 1,\cr \frac{w}{6}
+\frac{1}{3}, \,\,\,\,w\geq 1,\end{cases}$$ and
$$B_2(w)=\begin{cases} \frac{w}{3},\,\,\,\,\, 0\leq w \leq 1,\cr \frac{w}{9} +\frac{2}{9},\,\,\,
 w\geq 1.\end{cases}\,$$ Now we assume $g_1 =g_2\equiv 1$.
Thus we have
$$
 \frac{1}{m_{1}}=\frac{p_1-1}{p_1} +h_{12},
 $$
$$
 \frac{1}{m_{2}}=\frac{p_2-1}{p_2}\left(\frac{1}{2}\right)^{\frac{p_2}{p_2-1}}+h_{22}\,\left(\frac{1}{2}\right)^{\frac{1}{p_2-1}},
 $$
 $$
\frac{1}{M_1}=\frac{1}{M_1[0,b_1]}=\frac{p_1-1}{p_1}\,\,b_1^{\frac{p_1}{p_1-1}}+h_{11}\,b_1^{\frac{1}{p_1-1}}
$$
and
$$
 \frac{1}{M_2}=\frac{1}{M_{2}[a_2,b_2]}
 =\frac{1}{2}\min_{a_2\le \nu\le b_2}\left(\frac{p_2-1}{p_2}\,\,\left((\nu-a_2)^{\frac{p_2}{p_2-1}}+(b_2-\nu)^{\frac{p_2}{p_2-1}}\right)
 +h_{21}(\nu-a_2)^{\frac{1}{p_2-1}}\right).
$$
The choice $p_1=\frac{3}{2}$, $p_2=3$, $b_1=\frac{2}{3}$,
$a_2=\frac{1}{4}$ , $b_2=\frac{3}{4}$, $h_{11}=1/6$, $h_{12}=1/2$,
$h_{21}=1/9$ and $h_{22}=1/3$ gives by direct computation:
$$
c_1=\frac{1}{3};\ c_2=\frac{1}{4};\  m_1=1.2;\ M_1=5.78571;\
m_2=2.12132;\ M_2=9.14497.
$$
Let us now consider
$$
f_1(t, u,v)=\frac{1}{16}(u^4+t^3v^3)+\frac{27}{50}, \quad
f_2(t,u,v)=(tu)^{\frac{1}{2}}+10v^{9}.
$$

Then, with the choice of $\rho_1=\rho_2=1/20$, $r_1=1$, $r_2=2/3$,
$s_1=s_2=9$,  we obtain
$$\inf \Bigl\{ f_1(t,u,v):\; (t,u,v)\in [0,\frac{2}{3}]\times[0,\rho_1]\times[0,\rho_2] \Bigr\}= f_1(0,0,0)\\=0.54>\sqrt{M_1\rho_1}=0.538,$$
$$\sup \Bigl\{ f_1(t,u,v):\; (t,u,v)\in [0,1]\times [0,r_1]\times[0, r_2]\Bigr\}=f_1(1,r_1,r_2)=0.62< \sqrt{m_1 r_1}=1.095,$$
$$\inf \Bigl\{ f_1(t,u,v):\; (t,u,v)\in
[0,2/3]\times[c_1s_1,s_1]\times[0,s_2]\Bigr\}=f_1(0,c_1s_1,0)=5.602
>\sqrt{M_1 s_1}=1.247,$$
$$\sup \Bigl\{ f_2(t,u,v):\; (t,u,v)\in [0,1]\times [0,r_1]\times[0, r_2]\Bigr\}=f_2(1,r_1,r_2)=1.260<(m_2 r_2)^2=2,$$
$$\inf \Bigl\{ f_2(t,u,v):\; (t,u,v)\in
[\frac{1}{4},\frac{3}{4}]\times[0,s_1]\times[c_2s_2,s_2]\Bigr\}=f_2(t,0,c_2s_2)=14778.9
>(M_2 s_2)^2=6774.07.$$

Thus the conditions
$(\mathrm{I}^{0}_{\frac{1}{20},\frac{1}{20}})^{\star}$,
$(\mathrm{I}^{1}_{1,2/3})$ and $(\mathrm{I}^{0}_{9,9})$ are
satisfied; therefore the system (\ref{lap2})-(\ref{bc2}) has at
least two nontrivial solutions $(u_1,v_1)$ and $(u_2,v_2)$ such that $1/20<\|(u_1,v_1)\|\leq 1$ and $1<\|(u_2,v_2)\|\leq 9$.

\end{document}